\documentclass[11pt, a4paper, reqno]{amsart}
\usepackage[a4paper, margin=3.75cm]{geometry}
\usepackage{amssymb}
\usepackage{mathrsfs}

\usepackage[breaklinks]{hyperref}

\hypersetup{
unicode=true,
colorlinks=true,
linkcolor=blue,
citecolor=blue,
urlcolor=blue,
filecolor=blue,
bookmarksnumbered=true,
pdfstartview=FitH,
pdfhighlight=/N
}

\makeatletter
\renewcommand\section{\@startsection{section}{1}%
  \z@{.7\linespacing\@plus\linespacing}{.5\linespacing}%
  {\normalfont\Large\bfseries\centering}}
\makeatother

\theoremstyle{plain}
\newtheorem{theorem}{Theorem}
\newtheorem{lemma}{Lemma}

\newtheorem{corollary}{Corollary}

\theoremstyle{definition}

\newtheorem{remark}{Remark}

\numberwithin{equation}{section}

\def\ZZ{\mathbb Z}
\def\({\left(}
\def\){\right)}
\def\[{\left[}
\def\]{\right]}
\def\<{\left\langle}
\def\>{\right\rangle}

\def\mdet{\operatorname{det}}

\def\GL{\operatorname{GL}}

\def\mdeg{\operatorname{deg}}
\def\const{\mathrm{const}}

\def\NN{\mathbb N}
\def\CC{\mathbb C}
\def\QD{\mathrm{QD}}
\def\myk{\mathbf k}
\let\geq\geqslant

\begin{document}

\title{Viskovatov algorithm for Hermite--Pad\'e polynomials}

\author[Nikolay~R.~Ikonomov]{N.~R.~Ikonomov}
\address{Institute of Mathematics and Informatics, Bulgarian Academy of Sciences}
\email{nikonomov@math.bas.bg}
\author[Sergey~P.~Suetin]{S.~P.~Suetin}
\address{Steklov Mathematical Institute of the Russian Academy of Sciences}
\email{suetin@mi-ras.ru}
\thanks{The research of the second author was carried out with
partial financial support of the Russian Foundation for Basic Research (grant no.\ 18-01-00764).}

\date{July 7, 2020}

\begin{abstract}
We propose an algorithm for producing Hermite--Pad\'e polynomials of type~I
for an arbitrary tuple of $m+1$ formal power series $[f_0,\dots,f_m]$,
$m\geq1$, about $z=0$ ($f_j\in\CC[[z]]$) under the assumption that the series
have a~certain (`general position') nondegeneracy property.
This algorithm is a~straightforward extension
of the classical Viskovatov algorithm for construction of Pad\'e polynomials
(for $m=1$ our algorithm coincides with the Viskovatov algorithm).

The algorithm proposed here is based on a~recurrence relation and has the  feature
that all the Hermite--\allowbreak Pad\'e %!!!!!!
polynomials corresponding to the multiindices
$(k,k,k,\dots,k,k)$, $(k+1,k,k,\dots,k,k)$, $(k+1,k+1,k,\dots,k,k),\dots$,
$(k+1,k+1,k+1,\dots,k+1,k)$
are already known by the time
the algorithm produces the Hermite--Pad\'e polynomials corresponding to the multiindex
$(k+1,k+1,k+1,\dots,k+1,k+1)$.

We show how the Hermite--Pad\'e polynomials corresponding to different multiindices can be
found via this algorithm by changing appropriately the initial conditions.

The algorithm can be parallelized in $m+1$ independent evaluations at each $n$th step.

Bibliography:~\cite{Ziv16}~titles.

\end{abstract}

\maketitle

\markright{Viskovatov algorithm}

\setcounter{tocdepth}{1}
\tableofcontents

%%%fulltext

\section{The case $m=1$: a~tuple of series $[f_0,f_1]$ (finding Pad\'e polynomials)}\label{s1}

\subsection{Introduction. Classical Viskovatov algorithm}\label{s1s1}

The most well known algorithms for producing the coefficients of the expansion of a~given (in general, formal)
power series $f\in\CC[[z]]$ in a~continued $C$-fraction (this thereby also produces Pad\'e polynomials)
are the  $\QD$-algorithm
(see \cite[Ch.~4, \S\,4.3, Theorem 4.3.5]{BaGr96}) and
the Viskovatov algorithm (see \cite[\S\,4, pp.~243--245]{Vis03} and~\cite[Ch.~4, \S\,4.2, formula~(2.17)]{BaGr96}).
Both algorithms are capable of producing Pad\'e approximants of the form
$[n/n]_f,[n+1/n]_f$, but they can be applied only under certain nondegeneracy condition of the
original power series~$f$ (the series~$f$ should be in `general position').

In the present paper, we extend the classical Viskovatov algorithm
to the setting of Hermite--Pad\'e polynomials of type~I for a~tuple of $m+1$
formal power series $f_0,\dots,f_m$, where $m\geq1$ and
$f_j\in\CC[[z]]$, $j=0,\dots,m$. For $m=1$, for a~tuple of series $[f_0,f_1]$, this algorithm
(for finding Hermite--Pad\'e polynomials) produces Pad\'e polynomials.
This is why it can be looked upon as a~natural extension of the classical Viskovatov algorithm
to the case of Hermite--Pad\'e polynomials. It seems that with this algorithm one can associate the
so-called vector continued fraction (for details, see \cite[Ch.~8, \S\,8.4]{BaGr96},~\cite{MaTs17} and the references given therein).
However, this question will not be addressed in the present paper.
Recurrent relations close to those discussed below and related to the
extension of the Viskovatov algorithm to the case of Hermite--Pad\'e polynomials
were obtained in \cite{Ser86} and~\cite{MaTs17}; see also \cite{DeDi79},~\cite{BeLa92}, \cite{ShYi19} and~\cite{BeLa00}.

Note that it is the Viskovatov algorithm that ws used by A.~Trias when im\-ple\-ment\-ing
the HELM-algorithm for producing Pad\'e polynomials (see \cite{Tri17},~\cite{Tri17b}, and also~\cite{Tri18}).

Let $a(z):=\sum\limits_{k=0}^\infty a_kz^k$,
$b(z):=\sum\limits_{k=0}^\infty b_kz^k$ be formal power series,  $b_0\neq0$.
The Viskovatov algorithm of expansion of the ratio of the series
$a(z)/b(z)$ in a~continued $C$-fraction
is based on the following identity (see \cite[\S\,4.2, formula~(2.17)]{BaGr96}):
\begin{equation}
\frac{\sum\limits_{k=0}^\infty a_kz^k}{\sum\limits_{k=0}^\infty b_kz^k}
=\frac{a_0}{b_0}+\frac{z}{\sum\limits_{k=0}^\infty b_kz^k\bigm/
\sum\limits_{k=0}^\infty(a_{k+1}-a_0b_{k+1}/b_0)z^k}.
\label{1}
\end{equation}
An application of a~similar identity to the series appearing in the denominator on the right of~\eqref{1}
constitutes the next step in the expansion of the ratio of the series $a(z)/b(z)$
in a~continued $C$-fraction
\begin{equation}
v_0+\cfrac{z}{v_1+{\cfrac{z}{v_2+\cfrac{z}{v_3+\dots}}}}.
\label{2}
\end{equation}
In the Pad\'e table for the series $f(z):=a(z)/b(z)$,
the approximants of the continued
$C$-fraction~\eqref{2} form a~staircase sequence consisting of
Pad\'e approximants  of the form $[n/n]_f$ and $[n+1/n]_f$, $n=0,1,\dots$.

Note that for formal Laurent series about the point at infinity~$\zeta$ of the form
\begin{equation}
F(\zeta):=\sum_{k=0}^\infty\frac{c_k}{\zeta^{k+1}},
\label{3}
\end{equation}
expansions in a~Chebyshev continued fraction are constructed using the
Jacobi--\allowbreak Perron algorithm. There exists a~multivariate analogue of the
Jacobi--\allowbreak Perron algorithm capable of producing Hermite--Pad\'e polynomials for a~tuple of
series~\eqref{3}
(see \cite{NiSo88},~\cite{Par81} and the references cited there). Such algorithms also apply
for producing Hermite--Pad\'e polynomials under certain nondegeneracy conditions.
One can also mention $\QD$-algorithms (see, first of all~\cite{Ise87} and~\cite{Ise89} and the bibliography
given there)
capable of producing Hermite--Pad\'e polynomials for a~tuple of Laurent series~\eqref{3}
under similar nondegeneracy conditions.
However, these $\QD$-algorithms for construction of Hermite--Pad\'e polynomials
pertain to formal Laurent series rather than to formal Taylor series, and, as far as we know,
no link between formal Taylor series and the classical $\QD$-algorithms has been established so far.

It is also worth noting that from the point of view of  applications,
the traditional interest to Hermite--Pad\'e polynomials of type~I stems mainly from the
fact that they underlie the construction of Shafer quadratic approximants (also called algebraic approximants);
see \cite{Ser86}, \cite{SaToSu91}, \cite{SeGo98},
\cite{FeHo04},~\cite{KoKrPaSu16},~\cite{Ziv16}, \cite{AmBoFe18}, \cite{FaHaSpWe19} and the references there.
Nevertheless, a~new approach to the analytic continuation problem was recently proposed in
\cite[\S\,4, formulae (61)--(63)]{Sue18d} (see also \cite{Kom20}); this approach is based on the use of
Hermite--Pad\'e polynomials of type~I, but in it only rational functions are involved.
In~\cite{Kom20}, this approach was theoretically justified for a~sufficiently large class of
multivalued analytic functions; for more details, see \cite[formula~(9), Corollary~6]{Kom20}.

\subsection{Hermite--Pad\'e polynomials (Pad\'e polynomials) for a~tuple of series
$[f_0,f_1]$}\label{s1s2}
Here we reduce the Viskovatov algorithm to the form more convenient for further applications.
We set $f_0=f_0(z)=\sum\limits_{k=0}^\infty c_{0,k}z^k=c_0+O(z)$,
$f_1=f_1(z)=\sum\limits_{k=0}^\infty c_{1,k}z^k=c_1+O(z)$; here and in what follows,
by $O(z)$ we denote power series starting with the first-order term in~$z$.
Hence relation~\eqref{1} can be written as the identity
\begin{equation}
\frac{f_1}{f_0}=\frac{c_1}{c_0}+\cfrac{z}{f_0\bigm/\biggl[\biggl(f_1-\cfrac{c_1}{c_0}f_0
\biggr)/z\biggr]}
=\frac{c_1}{c_0}+\cfrac{z}{f^{[1]}_1/f^{[1]}_0},
\label{4}
\end{equation}
where we put
\begin{equation}
f^{[1]}_1:=f_0=:\sum_{k=0}^\infty c_{1,k}^{[1]}=c_1^{[1]}+O(z),\ \ 
f_0^{[1]}:=\frac1z\(f_1-\frac{c_1}{c_0}f_0\)
=\sum_{k=0}^\infty c_{0,k}^{[1]}=c_0^{[1]}+O(z).
\notag
\end{equation}
Similarly to \eqref{4}, for the new series $f_0^{[1]}$ and $f_1^{[1]}$, we set
\begin{equation}
\frac{f_1^{[1]}}{f_0^{[1]}}=\frac{c_1^{[1]}}{c_0^{[1]}}
+\cfrac{z}{f^{[2]}_1/f^{[2]}_0},
\label{5}
\end{equation}
where
\begin{align}
f_1^{[2]}:&=f_0^{[1]}=:\sum_{k=0}^\infty c_{1,k}^{[2]}z^k=c_1^{[2]}+O(z),
\notag\\
f_0^{[2]}:&=\frac1z\(f_1^{[1]}-\frac{c_1^{[1]}}{c_0^{[1]}}f_0^{[1]}\)
=\sum_{k=0}^\infty c_{0,k}^{[2]}z^k=c_0^{[2]}+O(z).
\notag
\end{align}
From \eqref{3} and~\eqref{5} we get the following relation, in which we put
$f_0^{[0]}:=f_0=:\sum\limits_{k=0}^\infty c_{0,k}^{[0]}z^k=c_0^{[0]}+O(z)$,
$f_1^{[0]}:=f_1=:\sum\limits_{k=0}^\infty c_{1,k}^{[0]}z^k=c_1^{[0]}+O(z)$,
\begin{equation}
\frac{f_1}{f_0}=\frac{f^{[0]}_1}{f^{[0]}_0}=\frac{c^{[0]}_1}{c^{[0]}_0}
+\cfrac{z}{\cfrac{c^{[1]}_1}{c^{[1]}_0}+\cfrac{z}{f^{[2]}_1/f^{[2]}_0}}.
\label{6}
\end{equation}
Applying formula of the form \eqref{5} to  $f^{[2]}_1/f^{[2]}_0,f^{[3]}_1/f^{[3]}_0,\dots$,
we get the (formal) expansion  of the ratio $f_1/f_0$ in a~continued  $C$-fraction
of the form~\eqref{2}.

\subsection{The matrix approach}\label{s1s3}
Let $f^{[0]}_0:=f_0,f^{[0]}_1:=f_1,f^{[1]}_0,f^{[1]}_1,\dots$ be formal series (in the above sense).

Setting
$$
M_1:=\begin{pmatrix}0&1\\1&0\end{pmatrix},\quad
M_2:=\begin{pmatrix}z&0\\-c^{[0]}_1/c^{[0]}_0&1\end{pmatrix},
\quad \vec{f}:=(f_1,f_0),
$$
we have
$$
M_1\binom{f_1}{f_0}=\binom{f_0}{f_1},\quad
M_2\binom{f_0}{f_1}=
\begin{pmatrix}zf_0\\f_1-\cfrac{c_1}{c_0}f_0\end{pmatrix}
=z
\begin{pmatrix}f_1^{[1]}\\f_0^{[1]}\end{pmatrix}.
$$
We define
\begin{equation}
M^{[0]}:=M_2M_1=
\begin{pmatrix}0&z\\1&-c^{[0]}_1/c^{[0]}_0\end{pmatrix}=
\begin{pmatrix}P_1&P_2\\Q_1&Q_2\end{pmatrix},
\label{6.2}
\end{equation}
where $P_1,P_2,Q_1,Q_2$ are polynomials of~$z$, $P_j.Q_j\in\CC[z]$.
So, we have
\begin{equation}
M^{[0]}\binom{f_1}{f_0}=z
\begin{pmatrix}f_1^{[1]}\\f_0^{[1]}\end{pmatrix}=O(z).
\label{7}
\end{equation}
Now from \eqref{7} we get that $Q_1f_1+Q_2f_0=O(z)$, where $Q_1=\const$, $
Q_2=\const$. Hence $Q_2$ and $Q_1$ are Hermite--Pad\'e polynomials of type~I (Pad\'e polynomials)
for the tuple of series [$f_0,f_1]$ and the multiindex  $\mathbf k=(0,0)$, because
the order of tangency in~\eqref{7} agrees with the multiindex~$\mathbf k$: $|\mathbf
k|+1=0+1=1$.

As in \eqref{7}, we have
$$
M^{[1]}
\begin{pmatrix} f_1^{[1]}\\f_0^{[1]}\end{pmatrix}
=z
\begin{pmatrix} f_1^{[2]}\\f_0^{[2]}\end{pmatrix},
\quad\text{where}\quad
M^{[1]}=
\begin{pmatrix}0&z\\1&-c_1^{[1]}/c_0^{[1]}\end{pmatrix}.
$$
Therefore,
\begin{equation}
M^{[1]}M^{[0]}\begin{pmatrix} f_1\\f_0\end{pmatrix}
=z^2
\begin{pmatrix} f_1^{[2]}\\f_0^{[2]}\end{pmatrix}=O(z^2).
\label{8}
\end{equation}
We set $A^{[0]}:=M^{[0]}$, $A^{[1]}:=M^{[1]}M^{[0]}$. Hence
\begin{equation}
A^{[1]}=M^{[1]}M^{[0]}
=
\begin{pmatrix}
z&-\dfrac{c_1}{c_0}z\\
-\dfrac{c_1^{[1]}}{c_0^{[1]}}&z+\cfrac{c_1^{[1]}}{c_0^{[1]}}
\dfrac{c_1}{c_0}
\end{pmatrix}
=:\begin{pmatrix}
P_1^{[1]}&P_2^{[1]}\\Q_1^{[1]}&Q_2^{[1]}
\end{pmatrix},
\label{8.2}
\end{equation}
where $\mdeg{Q_1^{[1]}}=0$, $\mdeg{Q_2^{[1]}}=1$. From \eqref{8}
it follows that $Q_2^{[1]}f_0+Q_1^{[1]}f_1=O(z^2)$. Therefore,
$Q_2^{[1]},Q_1^{[1]}$ is a~pair of Hermite--Pad\'e polynomials (Pad\'e polynomials) for the
tuple of series $[f_0,f_1]$
and the multiindex $\mathbf k=\mathbf k^{[1]}=(1,0)$.
Proceeding as in~\eqref{7} and~\eqref{8}, we arrive in succession to
Hermite--Pad\'e polynomials for the tuple  $[f_0,f_1]$ and the multiindices
$\mathbf k^{[2]}=(1,1)$, $\myk^{[3]}=(2,1)$, $\myk^{[4]}=(2,2)$, $\myk^{[5]}=(3,2),\dots$.

\subsection{Hermite--Pad\'e polynomials for a~tuple of
series $[f_0,f_1]$: an arbitrary iteration step $(n+1)$}\label{s1s4}
Let $n\geq1$. We set $$
f_1^{[n+1]}:=f_0^{[n]}=:c_1^{[n+1]}+O(z),\quad
f_0^{[n+1]}:=\frac1z\(f_1^{[n]}-\frac{c_1^{[n]}}{c_0^{[n]}}f_0^{[n]}\)
=:c_0^{[n+1]}+O(z).
$$
Let
$$
M^{[n]}:=
\begin{pmatrix}0&z\\1&-{c_1^{[n]}}/{c_0^{[n]}}\end{pmatrix},\quad
A^{[n]}:=M^{[n]}\cdots M^{[0]}=
\begin{pmatrix}P_1^{[n]}&P_2^{[n]}\\Q_1^{[n]}&Q_2^{[n]}\end{pmatrix},
$$
where $P_j^{[n]},Q_j^{[n]}$ are polynomials. Hence
\begin{equation}
A^{[n+1]}:=M^{[n+1]}A^{[n]}=
\begin{pmatrix}P_1^{[n+1]}&P_2^{[n+1]}\\Q_1^{[n+1]}&Q_2^{[n+1]}\end{pmatrix},
\label{10}
\end{equation}
where $P_j^{[n+1]},Q_j^{[n+1]}\in\CC[z]$. By definition of the matrix $M^{[n]}$
and using~\eqref{10}, we have the recurrence relations
\begin{equation}
\begin{aligned}
P_j^{[n+1]}&=zQ_j^{[n]},\\
Q_j^{[n+1]}&=P_j^{[n]}-\frac{c_1^{[n+1]}}{c_0^{[n+1]}}Q_j^{[n]},\quad
j=1,2,\dots,\quad n=1,2,\dots.
\end{aligned}
\label{11}
\end{equation}
From \eqref{11} we get the  {\it three-term} recurrence relation, which relates the
polynomials  $Q_j^{[n+1]}$, $j=1,2$, as obtained at step $(n+1)$ of the iteration,
to the polynomials obtained at two previous steps:
\begin{equation}
Q_j^{[n+1]}(z)=-\frac{c_1^{[n+1]}}{c_0^{[n+1]}}Q_j^{[n]}(z)+zQ_j^{[n-1]}(z),
\quad j=1,2,\quad n=1,2,\dotsc.
\label{12}
\end{equation}
The initial conditions for relations \eqref{12} follow from \eqref{6.2} and~\eqref{8.2}:
\begin{equation}
Q_1^{[0]}=Q_1=1,\quad Q_2^{[0]}=Q_2=-\frac{c_1}{c_0},\qquad
Q_1^{[1]}=-\frac{c_1^{[1]}}{c_0^{[1]}},
\quad Q_2^{[1]}=z+\frac{c_1^{[1]}}{c_0^{[1]}}\frac{c_1}{c_0}.
\label{12.2}
\end{equation}
Moreover, we have
\begin{equation}
A^{[n]}
\begin{pmatrix}f_1\\f_0\end{pmatrix}
=z^{n+1}\begin{pmatrix}f_1^{[n+1]}\\f_0^{[n+1]}\end{pmatrix}
=O(z^{n+1}).
\label{13}
\end{equation}

The following result holds.

\begin{theorem}\label{th1}
1) Let $n=2k$, $k\geq0$. Then $\mdeg Q_1^{[n]}=\mdeg Q_2^{[n]}=k$ and the polynomials
$Q_2^{[n]},Q_1^{[n]}$ are Hermite--Pad\'e polynomials for the tuple of series $[f_0,f_1]$
and the multiindex $\myk^{[n]}=(k,k)$ (the order of tangency is $|\myk^{[n]}|+1=2k+1=n+1$, see \eqref{13}).

2) Let $n=2k+1$, $k\geq0$. Then $\mdeg Q_1^{[n]}=k$, $\mdeg Q_2^{[n]}=k+1$ and the polynomials
$Q_2^{[n]},Q_1^{[n]}$ are Hermite--Pad\'e polynomials for the tuple of series $[f_0,f_1]$
and the multiindex $\myk^{[n]}=(k+1,k)$ (the order of tangency is
$|\myk^{[n]}|+1=2k+2=n+1$, see \eqref{13}).
\end{theorem}

\begin{remark}\label{rem0}
Here and in what follows, in Theorems \ref{th1}--\ref{th3} and
the corresponding algorithms, we assume that the original series $f_0,f_1,\dots,f_m$ are in `general position'
(cf.~\cite{Der94}). In particular, we assume that all the polynomials appearing in the statements of Theorems \ref{th1}--\ref{th3}
have the precise degree indicated in the theorems. In general, in the corresponding relations for the degree
the nonstrict inequality should be written.
\end{remark}

\begin{remark}\label{rem1}
If we start the iteration process from the vector series ${\,}^{\mathrm T\!}\!(f_0,f_1)$ (here  and in what follows
${\,}^{\mathrm T\!}\!\vec v$ denotes the transposition with respect to the
vector row
$\vec{v}$) instead of ${\,}^{\mathrm T\!}\!(f_1,f_0)$, then as a~result we get the sequence of Pad\'e approximants 
of the form $[n/n],[n/n+1]$, $n=0,1,\dotsc$ (instead of $[n/n],[n+1/n]$).
\end{remark}

Theorem \ref{th1} is in fact well known (see \cite[Ch.~4, Theorem  4.2.1]{BaGr96}; cf.\ also \eqref{12} and
\cite[formula~(2.76)]{BaGr96}). Nevertheless, to make our presentation complete, we give a~proof of this theorem.

\begin{proof}[Proof of Theorem~\ref{th1}] We argue by induction on~$k$.

1) Let $k=0$. From \eqref{12.2}, for $n=2k=0$ we get
$\mdeg{Q_1^{[0]}}=\mdeg Q_2^{[0]}=\nobreak 0$. Moreover, from \eqref{12} and~\eqref{12.2}
for $n=2k+1=1$ we have $\mdeg Q_1^{[1]}=0$, $\mdeg Q_2^{[1]}=\nobreak 1$. For $k=1$, using \eqref{12}
we find that  $\mdeg Q_1^{[2]}=\mdeg Q_2^{[2]}=1$ and $\mdeg
Q_1^{[3]}=1$, $\mdeg Q_2^{[3]}=2$.

2) Assuming now that the conclusions of Theorem~\ref{th1} hold for $n=2k-1$ and $n=2k$, let us
show that they also hold for $n=2k+1$ and $n=2k+2$.

For $n=2k+1$, from the recurrence relation \eqref{12} we get
$$Q_1^{[2k+1]}=a^{[2k+1]}Q_1^{[2k]}+zQ_1^{[2k-1]}.$$
Therefore,
$$
\mdeg Q_1^{[2k+1]}=\max\{\mdeg Q_1^{[2k]},1+\mdeg Q_1^{[2k-1]}\}
=\max\{k,1+k-1\}=k.
$$
A similar analysis shows that $Q_2^{[2k+1]}=a^{[2k+1]}Q_2^{[2k]}+zQ_2^{[2k+1]}$, and therefore,
using the induction assumption,
$$
\mdeg Q_2^{[2k+1]}=\max\{\mdeg Q_2^{[2k]},1+\mdeg Q_2^{[2k-1]}\}
=\max\{k,1+k\}=k+1.
$$

For $n=2k+2$ from \eqref{12} we conclude that
$Q_1^{[2k+2]}=a^{[2k+2]}Q_1^{[2k+1]}+zQ_1^{[2k]}$, and hence, by the induction assumption,
$$
\mdeg Q_1^{[2k+2]}=\max\{\mdeg Q_1^{[2k+1]},1+\mdeg Q_1^{[2k]}\}
=\max\{k,1+k\}=k+1.
$$
In a similar manner we obtain
$Q_2^{[2k+2]}=a^{[2k+2]}Q_2^{[2k+1]}+zQ_2^{[2k]}$, and therefore, by the induction assumption,
$$
\mdeg Q_2^{[2k+2]}=\max\{\mdeg Q_2^{[2k+1]},1+\mdeg Q_2^{[2k]}\}
=\max\{k+1,1+k\}=k+1.
$$

Theorem \ref{th1} is proved.
\end{proof}

\subsection{The algorithm for $m=1$ (a~tuple of series $[f_0,f_1]$)}\label{s1s5}
We are given two series
$f_0=f_0(z)=\sum\limits_{k=0}^\infty c_{0,k}z^k=c_0+O(z)$ and
$f_1=f_1(z)=\sum\limits_{k=0}^\infty c_{1,k}z^k=c_1+O(z)$.

\subsubsection{\bf The initial iteration step}\label{s1s5s1}
If $f_0^{[0]}:=f_0$, $f_1^{[0]}:=f_1$, $c_0^{[0]}:=c_0$, $c_1^{[0]}:=c_1$, then we have
$c_{0,k}^{[0]}:=c_{0,k}$, $c_{1,k}^{[0]}:=c_{1,k}$ for $k=0,1,\dots$, and $f_0^{[0]}=c_0^{[0]}+O(z)$, $f_1^{[0]}=c_1^{[0]}+O(z)$.
Let $a^{[0]}:=-c_1^{[0]}/c_0^{[0]}$.

\subsubsection{\bf Step $1$ of the iteration}\label{s1s5s2}
Setting
\begin{equation}
f_1^{[1]}:=f_0^{[0]},\quad f_0^{[1]}:=\frac1z\(f_1^{[0]}-\frac{c_1^{[0]}}{c_0^{[0]}}
f_0^{[0]}\),
\label{14}
\end{equation}
we have
$$
f_0^{[1]}=\sum_{k=0}^\infty c_{0,k}^{[1]}z^k=c_0^{[1]}+O(z),\quad
f_1^{[1]}=\sum_{k=0}^\infty c_{1,k}^{[1]}z^k=c_1^{[1]}+O(z).
$$
We set $a^{[1]}:=-c_1^{[1]}/c_0^{[1]}$,
$$
Q_1^{[0]}:=1,\quad Q_2^{[0]}:=a^{[0]},\quad
Q_1^{[1]}:=a^{[1]},\quad Q_2^{[1]}:=z+a^{[1]}a^{[0]}.
$$

\subsubsection{\bf Step $(n+1)$ of the iteration}\label{s1s5s3}
Let $a^{[n]}:=-c_1^{[n]}/c_0^{[n]}$,
\begin{equation}
\begin{aligned}
f_1^{[n+1]}:&=f_0^{[n]}=:\sum_{k=0}^\infty c_{1,k}^{[n+1]}z^k=c_1^{[n+1]}+O(z),\\
f_0^{[n+1]}:&=\frac1z\(f_1^{[n]}-\frac{c_1^{[n]}}{c_0^{[n]}}f_0^{[n]}\)
=\frac1z\(f_1^{[n]}+a^{[n]}f_0^{[n]}\)\\
&=\sum_{k=0}^\infty c_{0,k}^{[n+1]}z^k=c_0^{[n+1]}+O(z).
\end{aligned}
\label{15}
\end{equation}
Setting $a^{[n+1]}:=-c_1^{[n+1]}/c_0^{[n+1]}$, we find that  (see \eqref{12})
\begin{equation}
Q_j^{[n+1]}(z)=a^{[n+1]}Q_j^{[n]}(z)+zQ_j^{[n-1]}(z),
\quad j=1,2,\quad n=1,2,\dotsc.
\notag
\end{equation}

\section{The case $m=2$: a~tuple of series $[f_0,f_1,f_2]$}\label{s2}

\subsection{Introduction}\label{s2s1}
We are given three formal power series
$$ f_0=f_0(z)=\sum\limits_{k=0}^\infty c_{0,k}z^k ,\ 
f_1=f_1(z)=\sum\limits_{k=0}^\infty c_{1,k}z^k ,\ 
f_2=f_2(z)=\sum\limits_{k=0}^\infty c_{2,k}z^k .$$
If we define $f_0=c_0+O(z)$, $f_1=c_1+O(z)$, $f_2=c_2+O(z)$, then we get
\begin{align}
\begin{pmatrix}0&0&1\\1&0&0\\0&1&0\end{pmatrix}
\begin{pmatrix}f_2\\f_1\\f_0\end{pmatrix}
&=\begin{pmatrix}f_0\\f_2\\f_1\end{pmatrix},\label{20.2}\\
\begin{pmatrix}z&0&0\\0&1&-\dfrac{c_2}{c_1}\\-\dfrac{c_1}{c_0}&0&1\end{pmatrix}
\begin{pmatrix}f_0\\f_2\\f_1\end{pmatrix}
&=
\begin{pmatrix}zf_0\\f_2-\dfrac{c_2}{c_1}f_1\\f_1-\dfrac{c_1}{c_0}f_0
\end{pmatrix}
=z
\begin{pmatrix}f^{[1]}_2\\f^{[1]}_1\\f^{[1]}_0\end{pmatrix},
\label{21}
\end{align}
where
$$
f_2^{[1]}:=f_0,\quad f_1^{[1]}:=\frac1z\(f_2-\frac{c_2}{c_1}f_1\),\quad
f_0^{[1]}:=\frac1z\(f_1-\frac{c_1}{c_0}f_0\).
$$
We have
\begin{align}
f_0^{[1]}&=\sum_{k=0}^\infty c_{0,k}^{[1]}z^k=c_0^{[1]}+O(z),\notag\\
f_1^{[1]}&=\sum_{k=0}^\infty c_{1,k}^{[1]}z^k=c_1^{[1]}+O(z),\notag\\
f_2^{[1]}&=\sum_{k=0}^\infty c_{2,k}^{[1]}z^k=c_2^{[1]}+O(z).\notag
\end{align}
Let
\begin{gather}
M_1:=\begin{pmatrix}0&0&1\\1&0&0\\0&1&0\end{pmatrix},\quad
M_2:=
\begin{pmatrix}
z&0&0\\0&1&-c_2/c_1\\-c_1/c_0&0&1\end{pmatrix},
\notag\\
M^{[0]}:=M_2M_1=
\begin{pmatrix}0&0&z\\1&-c_2/c_1&0\\0&1&-c_1/c_0\end{pmatrix}
=\begin{pmatrix}P_1&P_2&P_3\\Q_1&Q_2&Q_3\\R_1&R_2&R_3\end{pmatrix}
\label{22}
\end{gather}
Hence relations \eqref{2} assume the form
\begin{equation}
M^{[0]}{\,}^{\mathrm T\!}\!\vec{f}=z{\,}^{\mathrm T\!}\!\vec{f}^{\,[1]},
\label{23}
\end{equation}
where $\vec{f}:=(f_2,f_1,f_0)$, $\vec{f}^{\,[1]}:=(f^{[1]}_2,f^{[1]}_1,f^{[1]}_0)$,
and ${\,}^{\mathrm T}\!\vec{v}$ is the transposition with respect to the
vector row   $\vec{v}=(v_1,v_2,v_3)$.

We set $f^{[0]}_0:=f_0=\sum\limits_{k=0}^\infty c_{0,k}^{[0]}z^k=c_0^{[0]}+O(z)$,
$f^{[0]}_1:=f_1=\sum\limits_{k=0}^\infty c_{1,k}^{[0]}z^k=c_1^{[0]}+O(z)$,
$f^{[0]}_2:=f_2=\sum\limits_{k=0}^\infty c_{2,k}^{[0]}z^k=c_2^{[0]}+O(z)$,
$a^{[0]}:=-c_2^{[0]}/c_1^{[0]}$, $b^{[0]}:=-c_1^{[0]}/c_0^{[0]}$.

\subsection{Theoretical results}\label{s2s2}
Given $n\geq1$, we define
\begin{equation}
\begin{aligned}
f_2^{[n+2]}:&=f_0^{[n]}=:\sum_{k=0}^\infty c_{2,k}^{[n+1]}z^k
=c_2^{[n+1]}+O(z),\\
f_1^{[n+1]}:&=\frac1z\(f_2^{[n]}-\frac{c_2^{[n]}}{c_1^{[n]}}f_1^{[n]}\)
=:\sum_{k=0}^\infty c_{1,k}z^k=c_1^{[n+1]}+O(z),\\
f_0^{[n+1]}:&=\frac1z\(f_1^{[n]}-\frac{c_1^{[n]}}{c_0^{[n]}}f_0^{[n]}\)
=:\sum_{k=0}^\infty c_{0,k}z^k=c_0^{[n+1]}+O(z).
\end{aligned}
\label{24}
\end{equation}
As in \eqref{22}, we put
\begin{equation}
M^{[n]}:=
\begin{pmatrix}
0&0&z\\1&-c_2^{[n]}/c_1^{[n]}&0\\0&1&-c_1^{[n]}/c_0^{[n]}
\end{pmatrix}.
\label{25}
\end{equation}
Let
$$
A^{[n]}:=M^{[n]}\cdots M^{[0]}=
\begin{pmatrix}
P_1^{[n]}&P_2^{[n]}&P_3^{[n]}\\
Q_1^{[n]}&Q_2^{[n]}&Q_3^{[n]}\\
R_1^{[n]}&R_2^{[n]}&R_3^{[n]}
\end{pmatrix}.
$$
Then
\begin{align}
A^{[n+1]}:&=
\begin{pmatrix}
P_1^{[n+1]}&P_2^{[n+1]}&P_3^{[n+1]}\\
Q_1^{[n+1]}&Q_2^{[n+1]}&Q_3^{[n+1]}\\
R_1^{[n+1]}&R_2^{[n+1]}&R_3^{[n+1]}
\end{pmatrix}=M^{[n+1]}A^{[n]}\notag\\
&=
\begin{pmatrix}
0&0&z\\1&-c_2^{[n+1]}/c_1^{[n+1]}&0\\0&1&-c_1^{[n+1]}/c_0^{[n+1]}
\end{pmatrix}
\begin{pmatrix}
P_1^{[n]}&P_2^{[n]}&P_3^{[n]}\\
Q_1^{[n]}&Q_2^{[n]}&Q_3^{[n]}\\
R_1^{[n]}&R_2^{[n]}&R_3^{[n]}.
\end{pmatrix}.
\label{26}
\end{align}
From \eqref{26} we get the  recurrence relations for $j=1,2,3$ and
for $n\geq1$
\begin{equation}
\begin{aligned}
P_j^{[n+1]}&=zR_j^{[n]},\\
Q_j^{[n+1]}&
=P_j^{[n]}-\frac{c_2^{[n+1]}}{c_1^{[n+1]}}Q_j^{[n]}
=a^{[n+1]}Q_j^{[n]}+P_j^{[n]},\\
R_j^{[n+1]}&=Q_j^{[n]}-\frac{c_1^{[n+1]}}{c_0^{[n+1]}}R_j^{[n]}
=b^{[n+1]}R_j^{[n]}+Q_j^{[n]},
\end{aligned}
\label{27}
\end{equation}
where
$$
a^{[n+1]}:=-\frac{c_2^{[n+1]}}{c_1^{[n+1]}},\quad
b^{[n+1]}:=-\frac{c_1^{[n+1]}}{c_0^{[n+1]}}.
$$
From \eqref{27} we finally get
\begin{equation}
\begin{aligned}
Q_j^{[n+1]}(z)&=a^{[n+1]}Q_j^{[n]}(z)+zR^{[n-1]}(z),\\
R_j^{[n+1]}(z)&=b^{[n+1]}R_j^{[n]}(z)+Q_j^{[n]}
\end{aligned}
\label{28}
\end{equation}
for $j=1,2,3$ and for $n=1,2,\dots$. Moreover, the initial conditions for the three-term
recurrence relations \eqref{28} follow from the equality $A^{[1]}=M^{[1]}M^{[0]}$ and read as
\begin{equation}
\begin{gathered}
Q_1^{[0]}=1, \quad Q_2^{[0]}=a^{[0]},\quad Q_3^{[0]}=0,\\
R_1^{[0]}=0,\quad R_2^{[0]}=1,\quad R_3^{[0]}=b^{[0]},\\
Q_1^{[1]}=a^{[1]},\quad Q_2^{[1]}=a^{[1]}a^{[0]},\quad Q_3^{[1]}=z,\\
R_1^{[1]}=1,\quad R_2^{[1]}=b^{[1]}+a^{[0]},\quad R_3^{[1]}=b^{[1]}b^{[0]}.
\end{gathered}
\label{28.2}
\end{equation}
Furthermore, from \eqref{25} we have
\begin{equation}
A^{[n]}
\begin{pmatrix}f_2\\f_1\\f_0\end{pmatrix}
=M^{[n]}\cdots M^{[0]}
\begin{pmatrix}f_2\\f_1\\f_0\end{pmatrix}
=z^{n+1}
\begin{pmatrix}f^{[n+1]}_2\\f^{[n+1]}_1\\f^{[n+1]}_0\end{pmatrix}
=O(z^{n+1}).
\label{29}
\end{equation}

The following result holds.

\begin{theorem}\label{th2}
Let $k=0,1,2,\dots$.

1) If $n=3k$, then
$\mdeg{Q_1^{[n]}}=k$, $\mdeg{Q_2^{[n]}}=k$, $\mdeg{Q_3^{[n]}}=k$.

2) If $n=3k+1$, then
$\mdeg R_1^{[n]}=k$, $\mdeg R_2^{[n]}=k$, $\mdeg R_3^{[n]}=k$,
$\mdeg{Q_1^{[n]}}=k$, $\mdeg{Q_2^{[n]}}=k$, $\mdeg{Q_3^{[n]}}=k+1$.

3) If $n=3k+2$, then
$\mdeg R_1^{[n]}=k$, $\mdeg R_2^{[n]}=k$, $\mdeg R_3^{[n]}=k+1$,
$\mdeg{Q_1^{[n]}}=k$, $\mdeg{Q_2^{[n]}}=k+1$, $\mdeg{Q_3^{[n]}}=k+1$.

4) If $n=3k+3$, then
$\mdeg R_1^{[n]}=k$, $\mdeg R_2^{[n]}=k+1$, $\mdeg R_3^{[n]}=k+1$.
\end{theorem}

From \eqref{29}, we obtain
\begin{equation}
(R_1^{[n]},R_2^{[n]},R_3^{[n]}){\,}^{\mathrm T\!}\!(f_2,f_1,f_0)=
R_1^{[n]}f_2+R_2^{[n]}f_1+R_3^{[n]}f_0=O(z^{n+1}).
\label{29.2}
\end{equation}

The next result follows from Theorem~\ref{th2} and~\eqref{29.2}.

\begin{corollary}\label{cor1}
Under the hypotheses of Theorem~\ref{th2}, for an arbitrary $n\in\NN$, let
$k_2=\mdeg{R_1^{[n]}}$, $k_1=\mdeg{R_2^{[n]}}$, $k_0=\mdeg{R_3^{[n]}}$, and
$\myk^{[n]}=(k_0,k_1,k_2)\in\ZZ_{+}^3$ (a~multiindex). Then
$|\myk^{[n]}|=k_0+k_1+k_2=n-1$, and by \eqref{29.2}, the tuple
$[R_3^{[n]},R_2^{[n]},R_1^{[n]}]=[Q_{\myk^{[n]},0},Q_{\myk^{[n]},1},
Q_{\myk^{[n]},2}]$ is a~tuple of Hermite--Pad\'e polynomials of type~I for the tuple of
functions $[f_0,f_1,f_2]$ and the multiindex $\myk^{[n]}$, the order of tangency being
$O(z^{|\myk^{[n]}|+2})=O(z^{n+1})$.
\end{corollary}

\begin{proof}
The proof of Theorem~\ref{th2} is by induction on~$k$.

I) Let $k=0$.

Assertions 1) and 2) of the theorem are direct consequences of representations \eqref{28.2}.

Using \eqref{28}, we find that
$Q_j^{[2]}=a^{[2]}Q_j^{[1]} +zR_j^{[0]}$, $j=1,2,3$.
By~\eqref{28.2}, we have $R_1^{[0]}=0$, $R_j^{[0]}=\const_j\neq0$, $j=2,3$.
Therefore,
from \eqref{28.2} we have $\mdeg{Q_1^{[2]}}=0$, $\mdeg Q_2^{[2]}=\mdeg
Q_3^{[2]}=1$.
Similarly, an appeal to \eqref{28} shows that $R_j^{[2]}=b^{[2]}R_j^{[1]}+Q_j^{[1]}$.
Therefore, by \eqref{28.2} we get $\mdeg R_1^{[2]}=\mdeg
R_2^{[2]}=0$, $\mdeg R_3^{[2]}=1$, proving assertion~3).

Let us verify assertion 4). From \eqref{28}, we conclude
$R_j^{[3]}=b^{[3]} R_j^{[2]}+Q_j^{[2]}$. Hence by the above properties of
the polynomials $R_j^{[2]}$ and $Q_j^{[2]}$, we find that  $\mdeg R_1^{[3]}=0$, $\mdeg
R_2^{[3]}=\mdeg R_3^{[3]}=1$.

This verifies all assertions 1)--4) of Theorem~\ref{th2} for $k=0$.

II) Now, assuming that assertions 1)--4) of Theorem~\ref{th2} hold for $n=3k$, $n=3k+1$, $n=3k+2$, and $n=3k+3$,
respectively, let us verify them under this hypothesis with $k$~replaced by $k+1$; that is,
for $n=3k+3$, $n=3k+4$, $n=3k+5$, and $n=3k+6$.

1) Let $n=3(k+1)=3k+3$. Then using~\eqref{28} and since assertions 2) and~3) of Theorem~\ref{th2}
hold for $n=3k+1$ and $n=3k+2$, we have
\begin{align}
\mdeg Q_1^{[3k+3]}&=\max\bigl\{\mdeg Q_1^{[3k+2]},1+\mdeg R_1^{[3k+1]}\bigr\}
=\max\{k,1+k\}=k+1,\notag\\
\mdeg Q_j^{[3k+3]}&=\max\{\mdeg Q_j^{[3k+2]},1+\mdeg R_j^{[3k+1]}\}
=\max\{k+1,1+k\}=k+1, \notag
\end{align}
where $j=2,3$.

2) Let $n=3k+4$. Then using~\eqref{28} and employing assertion~4) for $n=3k+3$ and assertion~1) for $n=3k+3$,
we have, for $j=1,2,3$,
$$
\mdeg R_j^{[3k+4]}=\max\{\mdeg R_j^{[3k+3]},\mdeg Q_j^{[3k+3]}\}=k+1.
$$
Similarly, using \eqref{28} and assertions~1) and~3) of Theorem~\ref{th2} for $n=3k+3$
and $n=3k+2$, we have, respectively,
\begin{align}
\mdeg Q_j^{[3k+4]}&=\max\{\mdeg Q_j^{[3k+3]},1+\mdeg R_j^{[3k+2]}\}=k+1
\quad\text{for}\quad j=1,2,\notag\\
\mdeg Q_3^{[3k+4]}&=\max\{\mdeg Q_3^{[3k+3]},1+\mdeg R_3^{[3k+2]}\}=k+2.
\notag
\end{align}

3) Let us verify assertion~3) for $n=3k+5$.
From \eqref{28} and assertion~2) for $n=3k+4$ we have, for $j=1,2$,
\begin{align}
\mdeg R_j^{[3k+5]}&=\max\{\mdeg R_j^{[3k+4]},\mdeg Q_j^{[3k+4]}\}
=\max\{k+1,k+1\}=k+1,\notag\\
\mdeg R_3^{[3k+5]}&=\max\{\mdeg R_3^{[3k+4]},\mdeg Q_3^{[3k+4]}\}
=\max\{k+1,k+2\}=k+2.\notag
\end{align}

Similarly, using \eqref{28} and employing the already proved assertions 1), 2) and~4)
of Theorem~\ref{th2} for $n=3k+3$, $n=3k+4$ and $n=3k+3$, respectively, we get
\begin{align}
\mdeg Q_1^{[3k+5]}&=\max\{\mdeg Q_1^{[3k+4]},1+\mdeg R_1^{[3k+3]}\}
=\max\{k+1,1+k\}=k+1,\notag\\
\mdeg Q_j^{[3k+5]}&=\max\{\mdeg Q_j^{[3k+4]},1+\mdeg R_j^{[3k+3]}\}
=k+2\quad\text{for}\quad j=2,3.\notag
\end{align}

4) Let us prove assertion~4) for $n=3(k+1)+3=3k+6$. From \eqref{28} and
since assertion~3) holds for $n=3k+5$, we find that
\begin{align}
\mdeg R_1^{[3k+6]}&=\max\{\mdeg R_1^{[3k+5]},\mdeg Q_1^{[3k+5]}\}=k+1,\notag\\
\mdeg R_j^{[3k+6]}&=\max\{\mdeg R_j^{[3k+5]},\mdeg Q_j^{[3k+5]}\}
=k+2\quad\text{for}\quad j=2,3.\notag
\end{align}

Theorem~\ref{th2} is proved.
\end{proof}

%\begin{comment}\end{comment}

\subsection{The algorithm for $m=2$ (a~tuple of series $[f_0,f_1,f_2]$)}\label{s2s3}
We are given three series:
$f_0=f_0(z)=\sum\limits_{k=0}^\infty c_{0,k}z^k=c_0+O(z)$,
$f_1=f_1(z)=\sum\limits_{k=0}^\infty c_{1,k}z^k=c_1+O(z)$, and
$f_2=f_2(z)=\sum\limits_{k=0}^\infty c_{2,k}z^k=c_2+O(z)$.

\subsubsection{\bf The initial iteration step: $n=0$}\label{s2s3s1}
We set $f_0^{[0]}:=f_0$, $f_1^{[0]}:=f_1$, $f_2^{[0]}:=f_2$,
$c_j^{[0]}:=c_j$, $j=0,1,2$,
$a^{[0]}:=-c_2^{[0]}/c_1^{[0]}$, $b^{[0]}:=-c_1^{[0]}/c_0^{[0]}$, and define
\begin{gather}
Q_1^{[0]}:=1,\quad Q_2^{[0]}:=a^{[0]},\quad Q_3^{[0]}:=0,\notag\\
R_1^{[0]}:=0,\quad R_2^{[0]}:=1,\quad R_3^{[0]}:=b^{[0]}.
\end{gather}

\subsubsection{\bf First iteration step: $n=1$}\label{s2s3s2}
We set
\begin{align}
f_2^{[1]}:&=f_0^{[0]}=:
\sum\limits_{k=0}^\infty c_{2,k}^{[1]}z^k=c_2^{[1]}+O(z),\notag\\
f_1^{[0]}:&=\frac1z\(f_2^{[0]}-\frac{c_2^{[0]}}{c_1^{[0]}}f_1^{[0]}\)
=\sum_{k=0}^\infty c_{1,k}^{[1]}z^k=c_1^{[1]}+O(z),\notag\\
f_0^{[0]}:&=\frac1z\(f_1^{[0]}-\frac{c_1^{[0]}}{c_0^{[0]}}f_0^{[0]}\)
=\sum_{k=0}^\infty c_{0,k}^{[1]}z^k=c_0^{[1]}+O(z).
\notag
\end{align}
Given $a^{[1]}:=-c_2^{[1]}/c_1^{[1]}$, $b^{[1]}:=-c_1^{[1]}/c_0^{[1]}$, we put
\begin{gather}
Q_1^{[1]}:=a^{[1]},\quad Q_2^{[1]}:=a^{[1]}a^{[0]},\quad Q_3^{[1]}:=z,
\notag\\
R_1^{[1]}:=1,\quad R_2^{[1]}:=b^{[1]}+a^{[0]},\quad R_3^{[1]}:=b^{[1]}b^{[0]}.
\notag
\end{gather}

\subsubsection{\bf Step $(n+1)$ of the iteration ($n\geq1$)}\label{s2s3s3}
Given  $n\geq1$, we put
\begin{align}
f_0^{[n+1]}:&=\frac1z\(f_1^{[n]}-\frac{c_1^{[n]}}{c_0^{[n]}}f_0^{[n]}\)
=\sum_{k=0}^\infty c_{0,k}^{[n+1]}z^k=c_0^{[n+1]}+O(z),\notag\\
f_1^{[n+1]}:&=\frac1z\(f_2^{[n]}-\frac{c_2^{[n]}}{c_1^{[n]}}f_1^{[n]}\)
=\sum_{k=0}^\infty c_{1,k}^{[n+1]}z^k=c_1^{[n+1]}+O(z),\notag\\
f_2^{[n+1]}:&=f_0^{[n]}=:\sum_{k=0}^\infty
c_{2,k}^{[n+1]}z^k=c_2^{[n+1]}+O(z),\notag\\
a^{[n+1]}&:=-\frac{c_2^{[n+1]}}{c_1^{[n+1]}},\quad
b^{[n+1]}:=-\frac{c_1^{[n+1]}}{c_0^{[n+1]}},\notag\\
Q_j^{[n+1]}(z)&=a^{[n+1]}Q_j^{[n]}(z)+zR_j^{[n-1]}(z),\notag\\
R_j^{[n+1]}(z)&=b^{[n+1]}R_j^{[n]}(z)+Q_j^{[n]}(z),\quad j=1,2,3.
\end{align}

\section{The case of an arbitrary  $m\in\NN$: a~tuple of series $[f_0,\dots,f_m]$}\label{s3}

\subsection{Introduction and theoretical results}\label{s3s1}
Given $(m+1)$ formal series
$f_j=f_j(z)=\sum\limits_{j=1}^\infty c_{j,k}z^k$, $f_j\in\CC[[z]]$,
$j=0,\dots,m$, we define $f_j=c_j+O(z)$. We have
\begin{gather}
\begin{pmatrix}
0&0&0&\dots&0&1\\
1&0&0&\dots&0&0\\
0&1&0&\dots&0&0\\
\hdotsfor6\\
0&0&0&\dots&1&0\end{pmatrix}
\begin{pmatrix}f_m\\f_{m-1}\\f_{m-2}\\\vdots\\f_0\end{pmatrix}
=
\begin{pmatrix}f_0\\f_m\\f_{m-1}\\\vdots\\f_1\end{pmatrix},
\label{31}\\
\begin{pmatrix}
z&0&0&0&0&\dots&0&0&0\\
0&1&-\frac{c_m}{c_{m-1}}&0&0&\dots&0&0&0\\
0&0&1&-\frac{c_{m-1}}{c_{m-2}}&0&\dots&0&0&0\\
\hdotsfor9\\
0&0&0&0&0&\dots&0&1&-\frac{c_2}{c_1}\\
-\frac{c_1}{c_0}&0&0&0&0&\dots&0&0&1
\end{pmatrix}
\begin{pmatrix}f_0\\f_m\\f_{m-1}\\\vdots\\f_1\end{pmatrix}
=
\begin{pmatrix}zf_0\\f_m-\frac{c_m}{c_{m-1}}f_{m-1}\\\vdots\\
f_2-\frac{c_1}{c_0}f_1\\f_1-\frac{c_1}{c_0}f_0\end{pmatrix}\notag\\
=z
\begin{pmatrix}f^{[1]}_m\\f^{[1]}_{m-1}\\f^{[1]}_{m-2}\\\vdots\\f^{[1]}_0
\end{pmatrix}=O(z),
\label{32}
\end{gather}
where
\begin{align}
f_m^{[1]}:&=f_0=:\sum_{k=0}^\infty c_{m,k}^{[1]}z^k=c_m^{[1]}+O(z),
\notag\\
f_j^{[1]}:&=\frac1z\(f_{j+1}-\frac{c_{j+1}}{c_j}f_j\)
=\sum_{k=0}^\infty c_{j,k}^{[1]}z^k=c_j^{[1]}+O(z),
\quad j=0,\dots,m-1.
\notag
\end{align}
We set $c_j^{[0]}:=c_j$,
$$
f_j^{[0]}:=f_j=:\sum_{k=0}^\infty c_{j,k}^{[0]}z^k=c_j^{[0]}+O(z),
\quad j=0,\dots,m,
$$
$a_j^{[0]}:=-c_{j}^{[0]}/c_{j-1}^{[0]}$, $j=1,\dots,m$, and define
\begin{gather}
M_1:=
\begin{pmatrix}
0&0&0&\dots&0&1\\
1&0&0&\dots&0&0\\
0&1&0&\dots&0&0\\
\hdotsfor6\\
0&0&0&\dots&1&0\end{pmatrix},
\notag\\
M_2:=
\begin{pmatrix}
z&0&0&0&0&\dots&0&0&0\\
0&1&a_m^{[0]}&0&0&\dots&0&0&0\\
0&0&1&a_{m-1}^{[0]}&0&\dots&0&0&0\\
\hdotsfor9\\
0&0&0&0&0&\dots&0&1&a_2^{[0]}\\
a_1^{[0]}&0&0&0&0&\dots&0&0&1
\end{pmatrix}.
\notag
\end{gather}
We also put
\begin{equation}
M^{[0]}:=M_2M_1=
\begin{pmatrix}
0&0&0&0&0&\dots&0&0&z\\
1&a_m^{[0]}&0&0&0&\dots&0&0&0\\
0&1&a_{m-1}^{[0]}&0&0&\dots&0&0&0\\
\hdotsfor9\\
0&0&0&0&0&\dots&1&a_2^{[0]}&0\\
0&0&0&0&0&\dots&0&1&a_1^{[0]}
\end{pmatrix}.
\label{31.2}
\end{equation}
We have
\begin{equation}
M^{[0]}{\,}^{\mathrm T\!}\!\vec{f}^{\,[0]}=z{\,}^{\mathrm T\!}\!\vec{f}^{\,[1]}=O(z),
\ \text{where}\ 
\vec{f}^{\,[0]}:=(f_0^{[0]},\dots,f_m^{[0]}),\quad
\vec{f}^{\,[1]}:=(f_0^{[1]},\dots,f_m^{[1]}).
\label{31.3}
\end{equation}

For $n\geq1$, we define
\begin{align}
a_j^{[n]}:&=-c_j^{[n]}/c_{j-1}^{[n]},\quad j=1,\dots,m,\notag\\
f_m^{[n+1]}:&=f_0^{[n]}=:\sum_{k=0}^\infty c_{m,k}^{[n+1]}z^k
=c_m^{[n+1]}+O(z),\notag\\
f_j^{[n+1]}:&=\frac1z\(f_{j+1}^{[n]}-\frac{c_{j+1}^{[n]}}{c_j^{[n]}}f_j^{[n]}\)
=\sum_{k=0}^\infty c_{j,k}^{[n+1]}z^k=c_j^{[n+1]}+O(z),
\  j=0,\dots,m-1.
\end{align}
Let
\begin{equation}
M^{[n]}:=
\begin{pmatrix}
0&0&0&0&0&\dots&0&0&z\\
1&a_m^{[n]}&0&0&0&\dots&0&0&0\\
0&1&a_{m-1}^{[n]}&0&0&\dots&0&0&0\\
\hdotsfor9\\
0&0&0&0&0&\dots&1&a_2^{[n]}&0\\
0&0&0&0&0&\dots&0&1&a_1^{[n]}
\end{pmatrix}.
\label{35}
\end{equation}
Then, as in \eqref{31.3}, we have, for $n\geq1$,
\begin{equation}
M^{[n]}{\,}^{\mathrm T\!}\!\vec{f}^{\,[n]}=z{\,}^{\mathrm
T\!}\!\vec{f}^{\,[n+1]},\quad\text{where}\quad
\vec{f}^{\,[n]}:=(f_m^{[0]},\dots,f_0^{[0]})\in\CC^{m+1}[[z]].
\label{36}
\end{equation}
From \eqref{36} it now follows that
$$
M^{[n]}\cdots M^{[0]}{\,}^{\mathrm T\!}\!\vec{f}^{\,[0]}=z^{n+1}{\,}^{\mathrm T\!}\!\vec{f}^{\,[n+1]}=O(z^{n+1}),
$$
and so
\begin{equation}
A^{[n]}{\,}^{\mathrm T\!}\!\vec{f}^{\,[0]}=z^{n+1}{\,}^{\mathrm T\!}\!\vec{f}^{\,[n+1]}=O(z^{n+1}),
\label{36.2}
\end{equation}
where $A^{[n]}:=M^{[n]}\cdots M^{[0]}$. Moreover, $A^{[0]}=M^{[0]}$, where
the matrix $M^{[0]}$ is given by \eqref{31.2}. Let
\begin{equation}
A^{[n]}=
\begin{pmatrix}\vec{A}_1^{[n]}\\\vdots\\\vec{A}_{m+1}^{[n]}\end{pmatrix},
\quad\text{where}\quad
\vec{A}_j^{[n]}:=(A_{j,1}^{[n]},\dots,A_{j,m+1}^{[n]}),
\quad j=1,\dots,m+1.
\label{36.3}
\end{equation}
By definition of the matrix $A^{[n]}$,
$$
A^{[n+1]}=M^{[n+1]}A^{[n]}.
$$
Therefore, an appeal to \eqref{35} yields
\begin{align}
\vec{A}_1^{[n+1]}&=z\vec{A}_{m+1}^{[n]},\label{37}\\
\vec{A}_j^{[n+1]}&=\vec{A}_{j-1}^{[n]}+a_{m+2-j}^{[n+1]}\vec{A}_j^{[n]}
=a_{m+2-j}^{[n+1]}\vec{A}_j^{[n]}+\vec{A}_{j-1}^{[n]},
\quad j=2,\dots,m+1.
\label{38}
\end{align}
Now using \eqref{37}, for $n\geq1$, we transform \eqref{38} to read
\begin{equation}
\begin{aligned}
\vec{A}_2^{[n+1]}&=a_m^{[n+1]}\vec{A}_2^{[n]}+z\vec{A}_{m+1}^{[n-1]},\\
\vec{A}_j^{[n+1]}&=a_{m+2-j}^{[n+1]}\vec{A}_j^{[n]}+\vec{A}_{j-1}^{[n]},
\quad j=3,\dots,m+1.
\end{aligned}
\label{39}
\end{equation}
Relations \eqref{39} are three-term recurrence relations
for $m$~rows $\vec{A}_j^{[n+1]}$, $j=2,\dots,m+1$, of the matrix $A^{[n+1]}$. The first row
is obtained from \eqref{37}. From \eqref{39} it follows that in order to
recursively find the entries of the rows
$\vec{A}_j^{[n+1]}$, $j=2,\dots,m+1$, of the matrix
$A^{[n+1]}$ it suffices to know the entries of the rows
$\vec{A}_j^{[1]}$, $j=2,\dots,m+1$, of the matrix $A^{[1]}$ and the $(m+1)$st row
$\vec{A}_{m+1}^{[0]}$ of the original matrix $A^{[0]}$. Let us find the required rows.

From \eqref{31.2} and since $A^{[0]}=M^{[0]}$ we get
$$
\vec{A}_{m+1}^{[0]}
=(\underbrace{0,0,0,0,\dots,0,1,a_1^{[0]}}_{m+1})\in\CC^{m+1}.
$$
Next, since $A^{[1]}:=M^{[1]}M^{[0]}$ and using~\eqref{35} (for $n=1$), we obtain
\begin{equation}
\begin{aligned}
&\vec{A}_2^{[1]}=(a_m^{[1]},a_m^{[1]}a_m^{[0]},0,0,\dots,0,0,z)\in\CC^{m+1},\\
&\vec{A}_j^{[1]}=(0,\dots,0,\underbrace{1,a_{m+2-j}^{[1]}+a_{m+3-j}^{[0]},
a_{m+2-j}^{[1]}a_{m+2-j}^{[0]}}_{j-2\,,j-1\,,j},0,\dots,0),
\,\, j=3,\dots,m+1,\\
&\vec{A}_{m+1}^{[1]}=(0,0,\dots,0,1,a_1^{[1]}+a_2^{[0]},a_1^{[1]}a_1^{[0]})
\in\CC^{m+1}.
\end{aligned}
\label{40}
\end{equation}
So, the initial conditions for the recurrence relations \eqref{39}
are given by $(m+1)$ vectors from the space $\CC^{m+1}$. By~\eqref{39}, all entries
$A_{j,k}^{[n]}$ are polynomials of~$z$,
$A_{j,k}^{[n]}\in\CC[z]$, for all $n=0,1,\dots$, and $j,k=1,\dots,m+1$.

Given an arbitrary vector $\vec{c}=(c_1,\dots,c_{m+1})\in\CC^{m+1}$, we set
$$
{\,}^{\mathrm B}\!\vec{c}:=(c_{m+1},\dots,c_1).
$$

As in \eqref{39}, we have
\begin{align}
A_{2,k}^{[n+1]}&=a_n ^{[n+1]}A_{2,k}^{[n]}+z A_{m+1,k}^{[n-1]},
\quad k=1,\dots,m+1,\notag\\
A_{j,k}^{[n+1]}&=a_{m+2-j}^{[n+1]}A_{j,k}^{[n]}+A_{j-1,k}^{[n]},
\quad k=1,\dots,m+1,\quad j=3,\dots,m+1.\notag
\end{align}

The following result holds.

\begin{theorem}\label{th3}
Let $m\in\NN$, $n\geq m-1$ and $n=m-1+(m+1)k+\ell$, where
$\ell\in\{0,\dots,m\}$, $k\in\{0,1,2,\dots\}$
($\ell=(n-m+1)\pmod{m+1}$). Then

1) If $\ell=0$, then $\mdeg A_{m+1,s}^{[n]}=k$ for all $s=1,\dots,m+1$.
The vector ${\,}^{\mathrm B}\!\vec{A}_{m+1}^{[n]}=\vec{Q}_{\myk^{[n]}}(\vec{f})
=(Q_{\myk^{[n]},0},Q_{\myk^{[n]},1},\dots,Q_{\myk^{[n]},m})$ is the vector of
Hermite--Pad\'e polynomials of type~I for the vector series
${\,}^{\mathrm B}\!\vec{f}=(f_0,\dots,f_m)$
and the multiindex  $\myk^{[n]}=(k,k,\dots,k)$. The order of tangency is
$|\myk^{[n]}|+m=(m+1)k+m=n+1$; that is,
\begin{equation}
{\,}^{\mathrm B}\vec{Q}_{\myk^{[n]}}(\vec{f}){\,}^{\mathrm T\!}\!\vec{f}=\vec{A}_{m+1}^{[n]}{\,}^{\mathrm T\!}\!\vec{f}=O(z^{n+1}).
\label{40.2}
\end{equation}

2) If  $\ell=1,\dots,m$, then
$\mdeg{A}_{m+1,m+1-s}^{[n]}=k+1$ for $s=0,\dots,\ell-1$ and
$\mdeg{A}_{m+1,m+1-s}^{[n]}=k$ for $s=\ell,\dots,m$.
Moreover, the vector ${\,}^{\mathrm B}\!\vec{A}_{m+1}^{[n]}=\vec{Q}_{\myk^{[n]}}(\vec{f})
=(Q_{\myk^{[n]},0},Q_{\myk^{[n]},1},\dots,Q_{\myk^{[n]},m})$ is the vector of
Hermite--Pad\'e polynomials of type~I for the vector function
${\,}^{\mathrm B}\!\vec{f}=(f_0,\dots,f_m)$
and the multiindex
$$\myk^{[n]}=(\underbrace{k+1,\dots,k+1}_{\ell} ,\ \underbrace{k,\dots,k}_{m+1-\ell}) .$$ %!!!!
The order of tangency is
$|\myk^{[n]}|+m=(m+1)k+m+\ell=n+1$; that is,
\begin{equation}
{\,}^{\mathrm B}\vec{Q}_{\myk^{[n]}}(\vec{f}){\,}^{\mathrm T\!}\!\vec{f}=\vec{A}_{m+1}^{[n]}{\,}^{\mathrm T\!}\!\vec{f}=O(z^{n+1}).
\label{40.3}
\end{equation}
\end{theorem}

So, for any $n\geq m-1$, $Q_{\myk^{[n]},j}=A_{m+1,m+1-j}^{[n]}$,
$j=0,\dots,m$, are Hermite--Pad\'e polynomials of type~I for the tuple $[f_0,\dots,f_m]$
and the multiindex $\myk^{[n]}=(k_0,\dots,k_m)$, where $k_j=
\mdeg{A_{m+1,m+1-j}^{[n]}}$, $j=0,\dots,m$, and the order of tangency is
$$
|\myk^{[n]}|+m=\sum_{j=0}^m \mdeg{A}_{m+1,m+1-j}^{[n]}+m=n+1
$$
(the multiindex $\myk^{[n]}$ is uniquely determined from the given number $n\geq m-1$).

\begin{remark}\label{rem20} Note that in the above recurrence algorithm
all tuples of
Hermite--\allowbreak Pad\'e polynomials corresponding to the multiindices of the form
$(k,k,k,\dots,k,k)$, $(k+1,k,k,\dots,k,k)$, $(k+1,k+1,k,\dots,k,k),\dots,
(k+1,k+1,k+1,\dots,k+1,k)$ are already evaluated
by the time when the tuple of Hermite--Pad\'e polynomials corresponding to the multiindex
$(k+1,k+1,k+1,\dots,k+1,k+1)$ is found.
According to~\cite{NiSo88}, such indices are called
{\it proper}.
\end{remark}

\begin{remark}[\rm (cf.\ Remark \ref{rem1})]\label{rem2}
If the iteration process is started from the vector function $\vec{f}^{\,[0]}=(f_0,\dots,f_m)\in\CC^{m+1}$
(instead of $\vec{f}^{\,[0]}=(f_m,\dots,f_0)$), then we get Hermite--Pad\'e polynomials of type~I
for the multiindices $(k,\dots,k,k,k),\dots$,
$(k,\dots,k,k,k+1)$, $(k,\dots,k,k+1,k+1),\dots$.

If we set $\vec{f}^{\,[0]}=(\underbrace{f_{s-1},\dots,f_0}_{s},f_m,\dots,f_s)$,
then the iteration process will produce Hermite--\allowbreak Pad\'e polynomials
for the multiindices (cf.~\cite{MaTs17}):
$$ (k,\dots,k) ,\ 
(k,\dots,k,\underbrace{k+1}_s,k,k,\dots,k) ,\ 
(k,\dots,k,\underbrace{k+1,k+1}_{s,s+1},k,\dots,k) .$$
\end{remark}

\begin{remark}\label{rem3} Hermite--Pad\'e polynomials are known to satisfy
many different recurrence relations, which are widely useful in theoretical studies
of Hermite--Pad\'e polynomials  (see,
for example, \cite{LoLo18}, \cite{LoMi18},~\cite{BaGeLo18} and the references given therein).
However, as a~rule, such results are obtained for a~system of Laurent series defined
at the point at infinity and are, in general, more involved than the three-term relations \eqref{39}.
\end{remark}

\begin{proof}[Proof of Theorem~\ref{th3}]
For vector $\vec{p}(z):=(p_1(z),\dots,p_{m+1}(z))\in\CC^{m+1}[z]$, where
$p_j(z)\in\CC[z]$, we define
$$
\mdeg\vec{p}(z):=(\mdeg p_1(z),\dots,\mdeg p_{m+1}(z))\in\ZZ^{m+1}_{+}.
$$
For vectors $\vec{p}(z)$ and $\vec{q}(z)$, $\vec{p},\vec{q}\in\CC^{m+1}[z]$,
the inequality $\mdeg\vec{p}\geq\mdeg\vec{q}$ means by definition that
$\mdeg\vec{p}-\mdeg\vec{q}\in \ZZ^{m+1}_{+}$. Similarly, the inequality $\mdeg\vec{p}>\mdeg\vec{q}$ means that
$\mdeg\vec{p}-\mdeg\vec{q}\in \NN^{m+1}$.

We set $\vec{e}_1:=(1,0,\dots,0)\in\ZZ^{m+1}$ and define
$\vec{e}_{m+1}:=(1,1,\dots,1)\in\NN^{m+1}$.

The following Lemma~\ref{lem1}, which holds under the hypotheses of Theorem~\ref{th3},
is the main ingredient in the proof of this theorem.

\begin{lemma}\label{lem1}
Let $n\geq m-1$. Then
\begin{equation}
\mdeg\vec{A}^{[n]}_{j-1}\geq \mdeg\vec{A}^{[n]}_{j}\quad\text{for}\quad
j=2,\dots,m.
\label{9.0}
\end{equation}
Let $n=m-1+(m+1)k+\ell$, where $\ell=(n-(m-1))\pmod{m+1}\in\{0,1,\dots,m\}$,
$k=0,1,\dots$. Then the following relations hold:
\begin{align}
\text{for all}\quad\ell\quad
\mdeg\vec{A}^{[n]}_{\ell+s}&
=(\underbrace{k,\dots,k}_s,\underbrace{k+1,\dots,k+1}_{m+1-s}),
\quad s=1,\dots,m+1-\ell,\label{9.1}\\
\text{for}\quad \ell\geq1\quad
\mdeg\vec{A}^{[n]}_{\ell-j}
&=(\underbrace{k+1,\dots,k+1}_{m+1-j},\underbrace{k+2,\dots,k+2}_j),\quad
j=0,\dots,\ell-1.
\label{9.2}
\end{align}
\end{lemma}

From relations \eqref{9.1} and~\eqref{9.2} of Lemma~\ref{lem1} for $s=m+1-\ell$ and $j=\ell-1$ for $\ell\geq1$ we
get, respectively, that
\begin{equation}
\begin{aligned}
\mdeg\vec{A}^{[n]}_{m+1}&=(\underbrace{k,\dots,k}_{m+1-\ell},
\underbrace{k+1,\dots,k+1}_{\ell}) \quad\text{and}\quad\\
\mdeg\vec{A}^{[n]}_1&=(\underbrace{k+1,\dots,k+1}_{m+2-\ell},
\underbrace{k+2,\dots,k+2}_{\ell-1}).
\label{9.3}
\end{aligned}
\end{equation}
Now \eqref{9.3} yields  $\mdeg{A}^{[n]}_{m+1,j}=k$ for $j=1,\dots,m+1-\ell$ and $\mdeg A^{[n]}_{m+1,j}=k+1$ for $j=m+1-\ell+1,\dots,m+1$.
These relations can be equivalently written  as
$\mdeg A^{[n]}_{m+1,m+1-s}=k+1$ for $s=0,\dots,\ell-1$ and $\mdeg A^{[n]}_{m+1,m+1-s}=k$ for $s=\ell,\dots,m$.
This implies assertions 1) and~2) of Theorem~\ref{th3}.

Theorem \ref{th1} is proved.
\end{proof}

\begin{proof}[Proof of Lemma~\ref{lem1}] We argue by induction on~$n$.

1) Let $n=m-1$. Then
\begin{equation}
A^{[n]}=M^{[n]}A^{[n-1]}=M^{[n]}\dotsb M^{[1]}M^{[0]},
\label{9.4}
\end{equation}
where the matrix $M^{[0]}$ has the form~\eqref{31.2}, and the matrix $M^{[n]}$ has the analogous form \eqref{35}.
Let $M^{[p]}$ be the matrix obtained from the matrix $M^{[n]}$
by replacing $n$ by~$p$, $p=0,1,\dots,n$. Then
$M^{[p]}\in\GL_{\CC}(m+1,m+1)$. Moreover, $\det M^{[p]}=(-1)^pz\not\equiv0$.

Let $B_2:=M^{[p]}B_1$, where $B_1\in\GL_{\CC}(m+1,m+1)$. The
matrix $B_2$ is the matrix~$B_1$ multiplied on the left by the matrix $M^{[p]}$.
The following facts are direct consequences of the structure \eqref{35} of the matrix  $M^{[p]}$:

1) the last $(m+1)$st row of the matrix $B_1$ is multiplied by~$z$ and becomes the
first row of the matrix $B_2$; so, the degree of the corresponding polynomial elements of the
original matrix $B_1$ is increased precisely by~1;

2) the second row of the matrix $B_2$ is obtained as follows: the second row of the matrix
$B_1$ is multiplied by the entry $a^{[p]}_m$ and is summed with an entry of the first row
of the matrix $B_1$;

3) for $j\geq2$, the corresponding $j$th row of the matrix $B_2$ is obtained as follows:
the $j$th row of the matrix $B_1$ is multiplied by the entry
$a^{[p]}_{m-j+2}$ and is summed with an entry of the $(j-1)$st row of the matrix $B_1$.

Hence, since now we now know the explicit form \eqref{31.2} of the matrix  $M^{[0]}$, it follows that
for $n=m-1$ the matrix $A^{[m-1]}=M^{[m-1]}\dotsb M^{[0]}$ reads as
\begin{equation}
\begin{pmatrix}
*&*z+*&*z+*&*z+*&\dots&*z+*&*z+*&*z+*\\
*&*&*z+*&*z+*&\dots&*z+*&*z+*&*z+*\\
*&*&*&*z+*&\dots&*z+*&*z+*&*z+*\\
\hdotsfor8\\
*&*&*&*&\dots&*&*&*z+*\\
c_m&c_{m-1}&c_{m-2}&c_{m-3}&\dots&c_2&c_1&c_0
\end{pmatrix},
\label{9.5}
\end{equation}
where  `$*$' and `$*z+*$' denote, respectively, some complex quantities and some first-degree
polynomials of~$z$ with complex coefficients. The precise values of these quantities and
the explicit form of these polynomials is immaterial here. The quantities $c_0,c_1,\dots,c_m$ are some complex
constant which are not simultaneously zero.\footnote{Here it is worth recalling that all
the original series $f_0,\dots,f_m$ are in general position. So,
$(c_0,\dots,c_m)\neq(0,\dots,0)$, inasmuch as
$\mdet M^{[p]}=(-1)^pz\not\equiv0$,
and similarly,
$\mdet A^{[p]}=(-1)^{pm}z\not\equiv0$.}
Moreover, from \eqref{36} and~\eqref{36.2} we have
\begin{equation}
c_0f_0+c_1f_1+\dots +c_{m-1}f_{m-1}+c_mf_m=O(z^m)=O(z^{n+1}).
\label{9.6}
\end{equation}
Relation \eqref{9.6} means that the nontrivial tuple of constant quantities
$[c_0,c_1,\dots$, $c_{m-1},c_m]$ %!!!!!!
is a~tuple of Hermite--Pad\'e polynomials of type~I
for the tuple of series $[f_0,f_1,\dots,f_{m-1},f_m]$ and
the multiindex
$\myk^{[m-1]}=(0,0,\dots,0,0)$ with the corresponding order of tangency, which is equal to
$|\myk^{[m-1]}|+m=m=n+1$.

Thus, representation \eqref{9.5} implies that the conclusions of Lemma~\ref{lem1}
with $n=m-1$ are true.

2) Assuming now that the conclusions of Lemma~\ref{lem1} hold for some
$n\geq m-1$, let us show that these results also hold with $n$ replaced by~$n+1$.

Let $n=m-1+(m+1)k+\ell$, $\ell=(n-m+1)\pmod{m+1}\in\{0,\dots,m\}$.

a) Let us prove inequality~\eqref{9.0}; that is, we need to show that
$\mdeg\vec{A}^{[n+1]}_j\geq\mdeg\vec{A}^{[n+1]}_{j+1}$ for $j=1,\dots,m$.

Since \eqref{9.3} holds by the induction assumption,
from \eqref{37} we have
$\mdeg\vec{A}^{[m+1]}_1\allowbreak =(k+1,\dots,k+1)$ for $\ell=0$ and
$\mdeg\vec{A}^{[m+1]}_1
=(\underbrace{k+1,\dots,k+1}_{m+1-\ell},\underbrace{k+2,\dots,k+2}_{\ell})$ for $\ell\geq1$. Moreover,
$\mdeg\vec{A}^{[n]}=(k,\underbrace{k+1,\dots,k+1}_m)$ for $\ell=0$ and
$\mdeg\vec{A}^{[n]}=(k+1,\dots,k+1)$ for $\ell=1$. As a~result,
$\mdeg\vec{A}^{[n+1]}_1>\mdeg\vec{A}^{[n]}_1$ for all $\ell$.
Using \eqref{9.1} and~\eqref{9.2}, which hold by the induction assumption,
we get  $\mdeg\vec{A}^{[n+1]}_1>\mdeg\vec{A}^{[n]}_1$.

%#######################################

Now an appeal to the recurrence relations~\eqref{38} shows that
\begin{equation}
\label{9.8}
\mdeg\vec{A}^{[n+1]}_1\geq\mdeg\vec{A}^{[n]}_1=\mdeg\vec{A}^{[n+1]}_2.
\end{equation}
The inequality
\begin{equation}
\label{9.9}
\mdeg\vec{A}^{[n+1]}_j\geq\mdeg\vec{A}^{[n+1]}_{j+1}, \quad j=2,\dots,m,
\end{equation}
follows from the recurrence relation~\eqref{38} and the induction assumption.
Hence from \eqref{9.8} and~\eqref{9.9} we get
$$
\mdeg\vec{A}^{[n+1]}_j\geq\mdeg\vec{A}^{[n+1]},\quad j=1,2,\dots,m.
$$

2) Let us now prove \eqref{9.1} and~\eqref{9.2} for the  next induction step; that is,
for  $n+1$. There are two cases to consider: $\ell\in\{0,\dots,m-1\}$ and $\ell=m$.

a) Let $\ell\in\{0,\dots,m-1\}$. Then $\ell+1\in\{1,\dots,m\}$ and $n+1=m-1+(m+1)k+\ell+1$. From \eqref{37} and~\eqref{9.1} for $s=m+1-\ell$
we get
\begin{align}
\mdeg\vec{A}^{[n+1]}_1 &=\mdeg\vec{A}^{[n+1]}_{m+1}+\vec{e}_{m+1}
=(\underbrace{k+1,\dots,k+1}_{m+1-\ell},\underbrace{k+2,\dots,k+2}_{\ell})\notag\\
&=(\underbrace{k+1,\dots,k+1}_{m+1-j},\underbrace{k+2,\dots,k+2}_{j})\bigr|_{j=\ell}
=\mdeg\vec{A}^{[n+1]}_{\ell'-j}\bigr|_{j=\ell'-1},\notag
\end{align}
where $\ell'=\ell+1$.
This proves formula~\eqref{9.2} for $n+1$ and $j=\ell=\ell'-1$.

By the induction assumption, $\mdeg\vec{A}^{[n]}_{j-1}\geq\mdeg\vec{A}^{[n]}_j$
 for $j=2,\dots,m+1$, and hence from \eqref{38} we have
\begin{equation}
\mdeg\vec{A}^{[n+1]}_j=\mdeg\vec{A}^{[n]}_{j-1},\quad j=2,\dots,m+1.
\label{9.10}
\end{equation}
Therefore,  for $\ell\in\{1,\dots,m-1\}$ and $\ell'=\ell+1\in\{2,\dots,m\}$ it follows by \eqref{9.1} that
\begin{align}
\mdeg\vec{A}^{[n+1]}_{\ell'+s}=\mdeg\vec{A}^{[n]}_{\ell+s}=(\underbrace{k,\dots,k}_{s},\underbrace{k+1,\dots,k+1}_{m+1-s}), \notag \\
\quad s=1,\dots,m+1-\ell'=1,\dots,m-\ell.
\label{9.11}
\end{align}
This proves \eqref{9.1} for the induction step $(n+1)$.

Let us verify \eqref{9.2}. By~\eqref{38} and using inequality \eqref{9.0}, we have
\begin{equation}
\mdeg\vec{A}^{[n+1]}_{\ell'-j}=\mdeg\vec{A}^{[n]}_{\ell'-j-1}=\mdeg\vec{A}^{[n]}_{\ell-j}
=(\underbrace{k+1,\dots,k+1}_{m+1-j},\underbrace{k+2,\dots,k+2}_{j}),
\label{9.12}
\end{equation}
where $j=0,\dots,\ell'-2=0,\dots,\ell-1$. For $j=\ell'-1=\ell$, we have
$\mdeg\vec{A}^{[n+1]}_{\ell'-j}=\mdeg\vec{A}^{[n+1]}_1$; \eqref{9.12} was proved above.

b)  Now let us assume that $\ell=m$. Then $n+1=m-1+(m+1)(k+1)$ and $\ell'=0$. In this case,
from \eqref{37} and~\eqref{9.3} we have  $s=m+1-s$ and
\begin{align}
\mdeg\vec{A}^{[n+1]}_1&=\vec{e}_1+\mdeg\vec{A}^{[n]}_{m+1}=
(k,\underbrace{k+1,\dots,k+1}_m)+\vec{e}_1
=(k+1,\underbrace{k+2,\dots,k+2}_m)\notag\\
&=\mdeg\vec{A}^{[n+1]}_{\ell'+s}\bigr|_{\ell'=0,s=1}
=(\underbrace{k+1,\dots,k+1}_s,\underbrace{k+2,\dots,k+2}_{m+1-s})\bigr|_{s=1}.
\notag
\end{align}
Next, for  $s=2,\dots,m+1$, $s-1=s'=m-j'$, we have
\begin{align}
\mdeg\vec{A}&^{[n+1]}_{\ell'+s}=\mdeg\vec{A}^{[n+1]}_s=\mdeg\vec{A}^{[n]}_{s-1}=\mdeg\vec{A}^{[n]}_{m-j'}
\notag\\
&=(\underbrace{k+1,\dots,k+1}_{m+1-j'},\underbrace{k+2,\dots,k+2}_{j'})
=(\underbrace{k+1,\dots,k+1}_{s},\underbrace{k+2,\dots,k+2}_{m+1-s}).
\notag
\end{align}
Hence \eqref{9.1} is proved in the $(n+1)$st step also for $\ell=m$.

Lemma~\ref{lem1} is proved.
\end{proof}

%#################################

\subsection{The algorithm for an arbitrary $m\in\NN$
(a~tuple of series $[f_0,\dots,f_m]$)}\label{s3s2}

\subsubsection{\bf Construction of new series $f_0^{[n]},\dots,f_m^{[n]}$,
$n=1,2,\dots$}\label{s3s2s1}
We are given $m\in\NN$ and series  $f_0,\dots,f_m\in\CC[[z]]$,
$f_j=f_j(z)=\sum\limits_{k=0}^\infty c_{j,k}z^k=c_j+O(z)$.

\smallskip
{\bf Initial step}.
We set
$f_j^{[0]}:=f_j=\sum\limits_{k=0}^\infty c_{j,k}z^k
=:\sum\limits_{k=0}^\infty c^{[0]}_{j,k}z^k=c_j^{[0]}+O(z)$,
$j=1,\dots,m$;
$a_j^{[0]}:=-c_j^{[0]}/c_{j-1}^{[0]}$, $j=1,\dots,m$.

\smallskip
{\bf Step $1$}.
We define
\begin{align}
f_m^{[1]}:&=f_0^{[0]}=:\sum_{k=0}^\infty c_{m,k}^{[1]}z^k=c_m^{[1]}+O(z),\notag\\
f_j^{[1]}:&=\frac1z\(f_{j+1}^{[0]}-\frac{c_{j+1}^{[0]}}{c_j^{[0]}}f_j^{[0]}\)
=\sum_{k=0}^\infty c_{j,k}^{[1]}z^k=c_j^{[1]}+O(z),
\quad j=0,\dots,m-1,\notag\\
a_j^{[1]}:&=-c_j^{[1]}/c_{j-1}^{[1]},\quad j=1,\dots,m.\notag
\end{align}

\smallskip
{\bf Step $(n+1)$ ($n\geq1$)}.
We set
\begin{align}
f_m^{[n+1]}:&=f_0^{[n]}=:\sum_{k=0}^\infty c_{m,k}^{[n+1]}z^k
=c_m^{[n+1]}+O(z)   ,\notag\\
f_j^{[n+1]}:&=\frac1z\(f_{j+1}^{[n]}+a_{j+1}^{[n]}f_j^{[n]}\)
=\sum_{k=0}^\infty c_{j,k}^{[n+1]}z^k=c_j^{[n+1]}+O(z),
\ \ j=0,\dots,m-1,\notag\\
a_j^{[n+1]}:&=-c_j^{[n+1]}/c_{j-1}^{[n+1]},\quad j=1,\dots,m.\notag
\end{align}

So, from a given tuple of series $f_0,\dots,f_m\in\CC[[z]]$, we have constructed new
series $f_0^{[n]},\dots,f_m^{[n]}\in\CC[[z]]$ and found the quantities
$a_1^{[n]},\dots,a_m^{[n]}\in\CC$ for all $n=0,1,2,\dots$.

\begin{remark}\label{rem4}
Note that the main purpose of this step in \S\,\ref{s3s2s1} is to find the
$m$ quantities $a_1^{[n]},\dots,a_m^{[n]}\in\CC$ ($n=1,2,\dots$).
\end{remark}

\begin{remark}\label{rem7}
By definition of the step of \S\,\ref{s3s2s1} the new series
$f^{[n+1]}_0,\dots,f^{[n+1]}_m$ can be evaluated from the series $f^{[n]}_0,\dots,f^{[n]}_m$ in parallel.
\end{remark}

\subsubsection{\bf Construction of Hermite--Pad\'e polynomials for a~multiindex
$\myk^{[n]}\in\ZZ_{+}^{m+1}$ {\rm(see \eqref{51})} for $n\geq m-1$}\label{s3s2s2} {\ }

\smallskip
{\bf Initial step}. We set
$$
\vec{A}_{m+1}^{[0]}:=(0,0,0,\dots,0,1,a_1^{[0]})\in\CC^{m+1}.
$$

\smallskip
{\bf Step $1$}. We set
\begin{align}
\vec{A}_2^{[1]}:&=(a_m^{[1]},a_m^{[1]}a_m^{[0]},0,0,\dots,0,0,z)\in\CC^{m+1},
\notag\\
&\quad\text{for $j=3,\dots,m+1$ we set }\notag\\
\vec{A}_j^{[1]}:&=(0,\dots,0,
\underbrace{1,a_{m+2-j}^{[1]}+a_{m+3-j}^{[0]},a_{m+2-j}^{[1]}a_{m+2-j}^{[0]}}_{j-2,\,\,j-1,\,\,j},
0,\dots,0)\in\CC^{m+1}.\notag
\end{align}

\smallskip
{\bf Step $(n+1)$, $n\geq m-1$}.
Setting $\ell:=(n-(m-1))\pmod{m+1}\in\{0,\dots,m\}$, we find~$k$ from the relation
$n-(m-1)=(m+1)k+\ell$
(we have $\ell=0,\dots,m$, $k:=(n-(m-1)-\ell)/(m+1)\in\ZZ_{+}$, $k=0,1,\dots$).

Next, we put
\begin{align}
\myk^{[n]}:&
=(\underbrace{k+1,\dots,k+1}_{\ell},\underbrace{k,\dots,k}_{m+1-\ell})\in\ZZ_{+}^{m+1},
\label{51}\\
\vec{A}_2^{[n+1]}:&=a_m^{[n+1]}\vec{A}_2^{[n]}+z\vec{A}_{m+1}^{[n-1]},
\notag\\
\vec{A}_j^{[n+1]}:&=a_{m+2-j}^{[n+1]}\vec{A}_j^{[n]}+\vec{A}_{j-1}^{[n]}, \quad j=3,\dots,m+1\notag.
\end{align}
Now, for all $n\geq m-1$, we have
$$
\vec{A}_{m+1}^{[n]}{\,}^{\mathrm T\!}\!\vec{f}=O(z^{n+1})
$$
and the vector $\vec{Q}_{\myk^{[n]}}(\vec{f}):
=(A_{m+1,m+1}^{[n]},\dots,A_{m+1,1}^{[n]})
=(Q_{\myk^{[n]},0},\dots,Q_{\myk^{[n]},m})$ is the vector of Hermite--Pad\'e polynomials of type~I
for the vector series
${\,}^{\mathrm B}\!\vec{f}:=(f_0,\dots,f_m)$ and the multiindex
$\myk^{[n]}=(k_0,\dots,k_m)$, where $k_j=
\mdeg{A_{m+1,m+1-j}^{[n]}}$, $j=0,\dots,m$, and the order of tangency is
$$
|\myk^{[n]}|+m=\sum_{j=0}^m \mdeg{A}_{m+1,m+1-j}^{[n]}+m=n+1
$$
(the multiindex $\myk^{[n]}$ is uniquely determined from the given number $n\geq m-1$; see~\eqref{51}).

%%%endfulltext

\end{document}